\documentclass[12pt]{article}
\usepackage{amssymb}
\usepackage{mathrsfs}
\usepackage{amsfonts}
\usepackage{graphicx}
\usepackage{mathptmx}

\usepackage{latexsym,amsmath,amssymb,amsfonts,amsthm}

\newcommand{\zb}{{\bar{z}}}
\newcommand{\pz}{\frac{\partial\ }{\partial z}}
\newcommand{\pzb}{\frac{\partial\ }{\partial \overline{z}}}
\newcommand{\R}{\mathbb{R}}
\newcommand{\h}{\mathbb{H}}
\newcommand{\s}{\mathbb{S}}

\newcommand{\B}{\mathbb{B}}

\newcommand{\C}{\mathbb{C}}
\newcommand{\N}{\mathbb{N}}

\newcommand{\meta}[2]{\langle #1,#2 \rangle }

\newcommand{\boP}{\mathcal{P}}

\newcommand{\boC}{\mathcal{C}}
\newcommand{\eps}{\epsilon}

\newcommand{\lan}{\langle}
\newcommand{\ran}{\rangle}

\DeclareMathOperator{\Tr}{Tr}

\newcommand{\Ome}{\Omega}

\newtheorem{theorem}{Theorem}[section]
\newtheorem{lemma}[theorem]{Lemma}

\newtheorem{corollary}[theorem]{Corollary}
\newtheorem{proposition}[theorem]{Proposition}
\newtheorem{remark}[theorem]{Remark}

\newtheorem{defn}[theorem]{Definition}

\newtheorem{claim}{Claim}

\newtheorem{theorema}{Theorem}

\begin{document}
\setcounter{page}{1}
\title{Characterization of $f$-extremal disks}
\author{Jos\'{e} M. Espinar$^{\dag}$, Laurent Mazet$^{*}$}
\date{}
\maketitle ~~~\\[-15mm]

\begin{center}
{\footnotesize $^{\dag}$Instituto Nacional de Matem\'{a}tica Pura e Aplicada, 110 Estrada
Dona Castorina, Rio de Janeiro, 22460-320, Brazil \\
Email: jespinar@impa.br\\}
{\footnotesize $^{*}$Universit\'e Paris-Est, Laboratoire d'Analyse et Math\'{e}matiques
Appliqu\'{e}es (UMR 8050), UPEM, UPEC, CNRS, F-94010, Cr\'{e}teil, France \\
Email: laurent.mazet@math.cnrs.fr\\}
\end{center}


\begin{abstract}
We show uniqueness for overdetermined elliptic problems defined on topological disks
$\Ome$ with $C^2$ boundary, \textit{i.e.}, positive solutions $u$ to
$\Delta u + f(u)=0$ in $\Omega \subset (M^2,g)$ so that $u = 0$ and $\frac{\partial
u}{\partial \vec\eta} = cte $ along $\partial \Omega$, $\vec\eta$ the unit outward normal along $\partial\Ome$ under the assumption of the existence of a candidate family. To do so, we adapt the G\'alvez-Mira generalized Hopf-type Theorem~\cite{GM} to the realm of overdetermined elliptic problem.

When $(M^2,g)$ is the standard sphere $\s^2$ and $f$ is a $C^1$ function so that $f(x)>0$ and $f(x)\ge xf'(x)$ for any $x\in\R_+^*$, we construct such candidate family considering rotationally symmetric solutions. This proves the Berestycki-Caffarelli-Nirenberg conjecture in $\s^2$ for this choice of $f$. More precisely, this shows that if $u$ is a positive solution to $\Delta u + f(u) = 0$ on a topological disk $\Omega \subset \s^2$ with $C^2$ boundary so that $u = 0$ and $\frac{\partial u}{\partial \vec\eta} =
cte $ along $\partial \Omega$, then $\Omega$ must be a geodesic disk and $u$ is rotationally symmetric. In particular, this gives a positive answer to the Schiffer conjecture D (cf. \cite{Sch,Shk}) for the first Dirichlet eigenvalue and classifies simply-connected harmonic domains (cf. \cite{RauSav}, also called {\it Serrin Problem}) in $\s ^2$.
\end{abstract}

\mbox{}\\
{\bf MSC 2010:}  35Nxx; 53Cxx. 
\\
{\bf Key Words:} Overdetermined Problems; Maximum principle; Neumann conditions; Index Theorem.

\markright{\sl\hfill  J.M. Espinar, L. Mazet  \hfill}


\section{Introduction}
\renewcommand{\thesection}{\arabic{section}}
\renewcommand{\theequation}{\thesection.\arabic{equation}}
\setcounter{equation}{0} \setcounter{maintheorem}{0}

Overdetermined elliptic problems (OEP), i.e., finding a solution to an elliptic partial
derivative equation constrained to both Dirichlet and Neumann conditions, appear
frequently in physical models and free boundary problems. 

Let $\Omega $ be an open connected domain of a complete connected Riemannian manifold
$(M,g)$ and consider the OEP given by 
\begin{equation}\label{eq:oep}
\begin{cases}
\Delta{u}+f(u)=0  &\text{ in }   \Omega,\\
u>0  				 &\text{ in }   \Omega,\\
u=0                     &\text{ on }\partial\Omega,\\
\langle\nabla{u},\vec{\eta}\rangle_{g}= \alpha &\text{ on } \partial\Omega,
\end{cases}
\end{equation}
where $\vec\eta$ is the unit outward normal vector along $\partial\Ome$, $\alpha $ a
negative constant and $f : \R \to \R $ is a continuous function. A domain $\Omega \subset M$ that supports a solution to \eqref{eq:oep} is called an {\it $f-$extremal domain}. If $\Omega \subset \R ^n$ (endowed with the standard Euclidean metric) is bounded and $f\equiv 1$, Serrin~\cite{s} proved that the ball is the only domain where the above problem admits a solution $u \in C^2(\Omega)$. This was generalized later to any Lipschitz function $f$ by Pucci and Serrin \cite{ps}. Serrin's proof uses the moving plane method introduced by Alexandrov in \cite{a} in order to prove that round spheres are the only constant mean curvature embedded hypersurfaces in $\R^n$. 

In $1997$, W. Reichel \cite{Rei} extended Pucci-Serrin result for exterior domains, that
is, connected smooth domains $\Omega \subset \R ^n $ such that $\R ^n \setminus \Omega$ is bounded, under the additional hypothesis that $f$ is non-increasing and the solution $u $ goes uniformly to a constant at infinity.  

In the same year , Berestycki, Caffarelli and Nirenberg \cite{bcn} considered the problem
\eqref{eq:oep} when $\Omega \subset \R ^n$ is an unbounded domain and also its complement,
such problem appears naturally in the regularity of free boundary solutions at a boundary
point. Under certain additional conditions they proved that the only $f-$extremal domain
whose boundary is an epigraph over a hyperplane is a half-space. So, combining the results of Pucci-Serrin, Reichel and Berestycki-Caffarelli-Nirenberg, they formulated the
following: 

\begin{quote}
\textbf{BCN conjecture:} \textit{If $f$ is Lipschitz, $\Omega \subset \R ^n$ is a smooth (in fact, Lipschitz) connected domain with $\R ^n \setminus \Omega$ connected where the OEP \eqref{eq:oep} admits a bounded solution, then $\Omega$ is either a ball, a half-space, a cylinder $\B ^k \times \R ^{n-k}$ ($\B^k$ is a ball of $\R^n$) or the complement of one of them.}
\end{quote}

P. Sicbaldi \cite{Sic} gave a counterexample of the BCN conjecture when $n \geq 3$.
Nevertheless, the BCN conjecture motivated interesting works as, for example, those of
Farina and collaborators (\cite{fv4,fv1,fv2,fv3} and references therein). Recently,
important contributions have been made in dimension $n=2$. First, Ros-Sicbaldi \cite{rs}
exploited the analogy between OEPs and constant mean curvature surfaces (in short, CMC
surfaces) which allowed them to prove the BCN conjecture in dimension $2$ under some extra hypothesis. Second, Ros-Ruiz-Sicbaldi \cite{RRS} proved that the BCN conjecture is true in dimension $2$ for unbounded domains whose complement is unbounded, such domain must be a half-space. Also, Ros-Ruiz-Sicbaldi \cite{RRS} constructed exteriors domains different from the exterior of a geodesic ball in $\R ^2$ for particular choices of the Lipschitz function $f$, this gives a counterexample to the BCN conjecture in $\R^2$ in all its generality. Hence, combining the works of Pucci-Serrin, Reichel and Ros-Ruiz-Sicbaldi we have

\begin{quote}
\textbf{Theorem \cite{ps,Rei,RRS}.} \textit{Let $f$ be a non-increasing Lipschitz function and $\Omega \subset \R ^2$ a $C^{2,\alpha}$ connected domain whose complement is connected. Assume that the OEP \eqref{eq:oep} admits a bounded solution $u$ that goes uniformly to a constant at infinity, then $\partial \Omega$ has constant curvature.}
\end{quote}

Observe that $\partial \Omega$ has constant curvature if, and only if, $\Omega$ is either
a ball, the exterior of a ball or a half-space. Recall that the hypothesis that $f$ is
non-increasing and $u \to C $ uniformly at infinity are only needed in Reichel's Theorem.
About the regularity of $\partial \Omega$, Pucci-Serrin and Reichel assumed $C^{2,\alpha}$ and Ros-Ruiz-Sicbaldi only Lipschitz. The above result is the best one can expect in this situation since the BCN conjecture is not true for any Lipschitz function $f$ (cf. \cite{RRS2}). In other words, we must assume additional conditions on $f$ and/or $u$ (see \cite{AftBus,Sir} for related conditions on $f$ and more general operators).

Regarding other Space Forms, combining the works of Molzon \cite{mr}, Espinar-Mao
\cite{EM} and Espinar-Farina-Mazet \cite{EFM}, we can prove the BCN conjecture for domains in the Hyperbolic space $\h^2$ under similar hypothesis than the Euclidean case, specifically: 

\begin{quote}
\textbf{Theorem \cite{EFM,EM,mr}.} \textit{Let $f$ be a non-increasing Lipschitz function and $\Omega \subset \h ^2$ a $C^{2,\alpha}$ connected domain whose complement is connected. Assume that the OEP \eqref{eq:oep} admits a bounded solution $u$ that goes uniformly to a constant at infinity, then $\partial \Omega$ has constant curvature.}
\end{quote}

Note that $\partial \Omega $ has constant curvature in $\h ^2$ if, and only if, $\Omega$
is either a geodesic disk, a horodisk, a half-space determined by a complete geodesic or
equidistant curve, or the complement of one of them. It would be interesting to construct a counterexample for exterior domains in $\h ^2$ in the spirit of \cite{RRS2}.

In the two dimensional sphere, the BCN conjecture reads as 
\begin{quote}
\textbf{BCN conjecture in $\s^2$:} \textit{If $f$ is Lipschitz, $\Omega \subset \s ^2$ is a
topological disk with $C^{2}$ boundary where the OEP \eqref{eq:oep} admits a solution,
then $\Omega$ is a geodesic ball.}
\end{quote}

Let us point out a couple of remarks. If we assume that $\Omega \subset \s ^2$ is connected and its complement also is connected then $\Omega $ is simply connected, in other words, $\Omega$ is a topological disk.

The previous mentioned works in $\R ^2$ and $\h ^2$ rely heavily in some variant of the
Alexandrov moving plane method introduced by Serrin in the context of OEP. We could also use this technique for domains $\Omega \subset \s ^2$ to prove the BCN conjecture in $\s ^2$ as far as $\Omega$ is contained in some hemisphere of $\s ^2$ (see the analogy with embedded CMC surfaces in $\s ^3$). 

Another capital result on CMC surfaces in $\R ^3$ is Hopf's Theorem that states that the
only compact immersed CMC surface of genus zero (a topological sphere) in $\R ^3$ is a
round sphere. H. Hopf gave two proofs of this result, one of them, the most interesting
for us, is based on the fact that for any CMC surface either the surface is totally
umbilic or there exists a line field with isolated singularities of negative index. So,
the Poincar\'{e}-Hopf index theorem eliminates the second possibility in a topological
sphere, thus the CMC surface must be a round sphere. Later, J. Nistche \cite{Nit}
extended Hopf's Theorem to compact immersed disks of constant mean curvature in $\R^3$
assuming that the boundary is a line of curvature, that is, it must be totally umbilic. If
it were not totally umbilic, in the interior of the disk, he can define the same line
field as Hopf and the condition on the boundary implies that the line field can be
extended continuously across the boundary after symmetrization of the domain, hence,
Nistche ended up with a line field in a sphere with isolated singularities of negative
index, which is impossible. Such technique depends heavily on the dimension but it has
been widely used in different geometric situations in order to classify topological
spheres (and compact disks under assumptions on the boundary) without assuming
embeddedness, a crucial hypothesis in the Alexandrov moving plane method (cf.
\cite{AR1,AR2,AEG1,AEG3,chern,EGR,HW,MMPR} and references therein). 

In a recent paper, J.A. G\'{a}lvez and P. Mira \cite{GM} proved an extremely general
version of Hopf's Theorem. Such version contains all the previous results mentioned here
among others, new ones and also applications to other problems as the Alexandrov
conjecture on the uniqueness of immersed spheres with prescribed curvature in $\R^3$. So, as Serrin did with the Alexandrov moving plane method, the main point of this paper is to adapt the G\'{a}lvez-Mira method to OEP, this can be seen as a Nistche type result for
OEP.

More precisely, the idea is to associate to each solution of the OEP~\eqref{eq:oep} a
traceless symmetric bilinear form. This bilinear form is defined in the support of the
solution and it either vanishes identically or has isolated zeroes. In the case it has
isolated zeroes, a traceless symmetric bilinear form defines a Lorentzian metric away from
its zeroes and, it is well-known, induces two line fields with singularities at the zeroes.
First, we show that the singularities of the line fields are of negative order and second,
the boundary is a "line of curvature" of this Lorentzian metric. This impose restrictions
on the geometry of the support of the solution.

Actually, in order to define such a bilinear form, we need to assume the existence of
certain solutions to \eqref{eq:oep}: the family of candidate solutions. In the case of
$\s^2$, we prove that such families of candidate solutions exist if $f$ satisfies some
hypothesis and then the BCN conjecture follows for these particular $f$.

The paper is organized as follows. In Section~\ref{sec:candfam}, we define what is a
candidate family of solutions. Constructions of some of them on $\s^2$ are given in
Section~\ref{Examples}. In Section~\ref{sec:mainth}, we state and prove the main theorem
concerning the existence of our traceless bilinear form. Then we apply it to the proof of
the BCN conjecture.

Just after writing our paper, we learned that P. Mira \cite{Mira} has also proved a result concerning to fully non-linear overdetermined elliptic problems in topological disks in $\R^2$. Nevertheless, the main result (cf. \cite[Theorem 2.4]{Mira}) does not apply in our situation since the functional $F[u]=0$ is not allow to depend on the base point $(x,y)\in \Omega$.


\section{Family of candidate solutions}\label{sec:candfam}

First, we need to define what is a smooth family of functions parametrized by some
manifold $N$. So let $N$ and $M$ be two manifolds; $N$ may have some boundary and $M$ is
endowed with some Riemannian metric $g$ whose Levi-Civita connection is denoted by
$\nabla$ (we will use the notation ($g(\cdot,\cdot)=\lan \cdot,\cdot\ran$). 

A domain in $M$ is just a connected open subset of $M$ with a $C^2$ boundary. For any
$p\in N$, we choose a domain $\Ome_p$ in $M$ in a continuous way. $(\Ome_p)_{p\in N}$ is
continuous if, for any $p_0\in N$, any closed subset $F$ of $\Ome_{p_0}$ and any open
subset $U$ with $\overline \Ome_{p_0}\subset U$, we have $F\subset \Ome_p\subset U$ for any $p$
close to $p_0$. Let us then denote
$$
A=\{(p,q)\in N\times M\,|\, \, q\in \overline\Ome_p\}
$$
\begin{defn}
With the above notations, a \emph{smooth family of $C^3$ functions} on $(\Ome_p)_{p\in N}$
is a map $V:A\to\R$ such that
\begin{itemize}
\item the function $v_p :\overline\Ome_p\to \R;\, q\mapsto V(p,q)$ is $C^3$,
\item there is an open neighborhood $B$ of $A$ in $N\times M$ such that $V$ can
be extended to $B$ and the map
$$
\Phi : \begin{matrix}
B&\to & \R\times TM\oplus S^2M\\
(p,q)&\mapsto& (v_p(q),(q,\nabla v_p(q),\nabla^2 v_p(q)))
\end{matrix}
$$
is a $C^1$ map on $A$.
\end{itemize}

Here $S^2M$ denotes the bundle of symmetric $2$-forms on $M$ and $\nabla^2$ is the hessian operator.
\end{defn}

Let us remark that in the following we will say that $(v_p)_{p\in N}$ is a smooth family of
functions. Moreover if $(\Ome_p')_{p\in N}$ is a continuous family of domains in $M$ such that
$\overline\Ome_p\subset \Ome_p'$ and $N$ has no boundary, then $A'=\{(p,q)\in N\times M\,|\,q\in
\overline\Ome_p'\}$ is a neighborhood of $A$.

Let us now fix a Lipschitz function $f$ and a connected Riemannian surface $(M,g)$ . Let
us define $N=(TM\times \R_+)\setminus\{(q,0,0)\in TM\times \R;\ q\in M\}$.

\begin{defn}\label{Def:Candidate}
We say that the OEP 
\begin{equation}\label{eq:oep2}
\begin{cases}
\Delta{u}+f(u)=0  &\text{ in }   \Omega \subset M,\\
u>0  				 &\text{ in }   \Omega,\\
u=0                     &\text{ on }\partial\Omega,\\
\langle\nabla{u},\vec{\eta}\rangle=\alpha&\text{ on } \partial\Omega,
\end{cases}
\end{equation}
admits a \emph{family of candidate solutions}, $\mathcal C$, if there is a smooth family
of $C^3$ functions $(v_{q,w,a})_{(q,w,a)\in N}$ where $\Omega _{q,w,a}\subset M$ has
$C^2$ boundary and $v_{q,w,a}$ solves
$$
\begin{cases}
\Delta{v_{q,w,a}}+f(v_{q,w,a})=0  &\text{ in }   \Omega_{q,w,a} ,\\
v_{q,w,a}>0  				 &\text{ in }   \Omega_{q,w,a},\\
v_{q,w,a} =0                     &\text{ on }\partial\Omega_{q,w,a},\\
\langle\nabla{v_{q,w,a}},\vec{\eta}\rangle= cte < 0 &\text{ on } \partial\Omega_{q,w,a},
\end{cases}
$$
such that
\begin{itemize}
\item[(a)] $q\in \overline\Ome_{q,w,a}$,
\item[(b)] $v_{q,w,a}(q) = a$,
\item[(c)] $\nabla v_{q,w,a}(q) =w$,
\item[(d)] if $(q_0,w_0,a_0)\in N$ and denote $\Ome_0=\Ome_{q_0,w_0,a_0}$ and
$v_0=v_{q_0,w_0,a_0}$, then, for any $q\in \overline\Ome_0$, we have
$$
\Ome_{q,\nabla v_0(q),v_0(q)}=\Ome_0\ \textrm{ and }\ v_{q,\nabla v_0(q),v_0(q)}=v_0 .
$$
\end{itemize}
\end{defn}

Let us remark that item (b) implies that $q \in \Omega _{q,w,a}$ if $a>0$ and $q \in
\partial \Omega_{q,w,a}$ if $a=0$. Item (d) is a uniqueness property of the family of
candidate solution. Actually, the same solution appears to be parametrized by several
$(q,w,a)\in N$. 

In the family of candidate solutions $\mathcal C$, are included all solutions with
negative constant Neumann boundary values, not only those with fixed Neumann boundary
values $\alpha$. Such subset of $\mathcal C$, those with Neumann boundary values $\alpha$,
is denoted by $\mathcal C _\alpha \subset \mathcal C$. In the proof of our main theorem,
we will see that only a part of $\boC$ will be used.


\section{Two examples}\label{Examples}

\textit{A priori}, constructing a family of candidate solutions is not easy to do. When we
consider a sufficiently symmetric space, say for example $\R^n$, the Alexandrov moving
plane method introduced by Serrin tells us that if $\Ome$ is a bounded domain which admit
a solution to the OEP~\eqref{eq:oep2} then $\Ome$ has to be a ball and $u$ is rotationally
symmetric. This suggest that at least in $\R^2$, $\s^2$ and $\h^2$, candidate solutions
can be constructed by considering rotationally symmetric solutions to \eqref{eq:oep2}.

Actually, in this section, we construct two examples of families of candidate solutions in
$\s^2$. The work can also be done in $\R^2$ and $\h^2$ but only $\s^2$ is interesting with
respect to our main theorem (see Corollary~\ref{cor:bcn}).


\subsection{Example 1}\label{Sub:Ex1}

Let $\lambda\in \R$, we consider $f(t)=\lambda t$ and the following OEP in $\Ome\subset\s^2$
\begin{equation}\label{eq:oep3}\tag{$\text{OEP}_\lambda$}
\begin{cases}
\Delta{u}+\lambda u=0  &\text{in }   \Omega,\\
u>0  				 &\text{in }   \Omega,\\
u=0                     &\text{on }\partial\Omega,\\
\langle\nabla{u},\vec{\eta}\rangle=\alpha < 0&\text{on } \partial\Omega.
\end{cases}
\end{equation}

First we remark that the existence of a solution implies $\lambda>0$ by the maximum
principle. So we focus on that case. The main result of this section is 
\begin{proposition}\label{prop:reparam}
For any $\lambda>0$, \eqref{eq:oep3} admits a family of candidate solutions.
\end{proposition}

The end of the Subsection \ref{Sub:Ex1} is devoted to the proof of the above proposition.

Let $\lambda$ be positive, then there is $R_\lambda\in(0,\pi)$ such that the first
eigenvalue of $-\Delta$ on a geodesic disk of radius $R_\lambda$ in $\s^2$ is $\lambda$
(see Theorem II.5.6 in \cite{Cha}). So if $p\in\s^2$ and $D_p$ is the geodesic disk in
$\s^2$ of center $p$ and radius $R_\lambda$, we have a solution $u_p$ to
$$
\begin{cases}
\Delta u+\lambda u=0&\text{ in }D_p,\\
u=0&\text{ on }\partial D_p,\\
u(p)=1.
\end{cases}
$$
Actually, $u_p$ is invariant by rotation around $p$ so $\lan\nabla
u_p,\vec\eta\ran=\alpha_\lambda$ on $\partial D_p$: $u_p$ is a $C^3$ solution to
\eqref{eq:oep3}. We notice that $R_\lambda$ is decreasing in $\lambda$ from $\pi$ to $0$.

If $p$ is in $\s^2$ and $q$ is a point at distance $\rho$ from $p$ (with $\rho<\pi$), the
function $u_p$ is just a function $U$ of $\rho$. Moreover $U$ solves the ODE 
\begin{equation}\label{eq:ode}
U''+(\cot\rho) U'+\lambda U=0.
\end{equation}
We remark that $U''(0)=-\frac\lambda2$.
Let us notice that if $L$ is an isometry of $\s^2$, $D_{L(p)}=L(D_p)$ and
$u_{L(p)}=u_p\circ L^{-1}$.

The family $(t u_p)_{t\in\R_+^*, p\in \s^2}$ is then a smooth family of functions, each of
them is a solution to \eqref{eq:oep3}. But it is not a family of candidate solutions as in
the definition since we do not have the right parametrization. The rest of the section is
devoted to prove we can reparametrize the family.

So let us consider $\eps>0$ such that $\bar R=R_\lambda+\eps<\pi$. $\eps$ will be fixed by
Lemma~\ref{lem:concav} below. First, the definition of $u_p$ extends to the geodesic disk
$\Delta_p$ of center $p$ and radius $\bar R$ by solving \eqref{eq:ode} up to $\bar R$. The
choice of $\eps$ is given by the following lemma.

\begin{lemma}\label{lem:concav}
There is $\eps>0$ such that $U''U-U'^2$ does not vanish on $(0,\bar R)$.
\end{lemma}



The proof will be given later. The parameter $\eps$ is now fixed. Actually one has
$U''U-U'^2<0$ in $(0,\bar R)$ (this is the log-concavity of $U$ on $(0,R_\lambda)$) and
even in $(-\bar R,\bar R)$ (by defining $U(-\rho) = U(\rho)$). The consequence of this is
that $U'<0$ on $(0,\bar R)$. It also implies that the curve 
$$\rho\in(-\bar R,\bar R)\mapsto (U(\rho),U'(\rho))\in\R^2\setminus\{(0,0)\}$$can be
described as a polar curve $A(\rho)(\cos\beta(\rho),\sin\beta(\rho))$ with $\beta(-\bar
R)=\bar\beta\in(\pi/2,\pi)$. The last claim follows since $U(R_\lambda )=0$ and $U ' <0$
in $(0, \bar R)$.

Let $S$ be the angular sector $\{(x,y)\in\R^2\, |\, \,x>\sqrt{x^2+y^2}\cos\bar\beta\}$ and consider the map
$$
\begin{matrix}
F: & \R_+^*\times (-\bar R,\bar R)&\to&  S\\
 & (t,\rho)&\mapsto&(tU(\rho),tU'(\rho))\end{matrix}
\,.
$$
The map $F$ is $C^1$ and the above discussion about the curve
$\rho\mapsto(U(\rho),U'(\rho))$ implies that $F$ is a bijection. Let us denote
$F^{-1}(x,y)=(T(x,y),R(x,y))$ and define 
$$\widetilde N=\{((q,w),a)\in T\s^2\times\R \, | \,\, (a,|w|)\in S\}.$$ 

Since $\nabla u_{t,p}(q)=tU'(\rho)\partial_\rho$, the distance $\rho$ can be computed as
$\rho=R(u_{t,p}(q),-|\nabla u_{t,p}|(q))$ and the parameter $t$ is $T(u_{t,p}(q),-|\nabla
u_{t,p}|(q))$. Let us now define 
$$G:(q,w,a)\in \widetilde N\mapsto \exp_q(\frac{ R(a,-|w|)}{|w|}w) .$$ 

Then a right parametrization of the family of candidate solutions is given on $\widetilde N$ by
$$
v_{q,w,a}=T(a,-|w|)u_{G(q,w,a)} .
$$
So, if we prove that $T$ and $G$ are $C^1$, we obtain a true family of candidate
solutions. This smoothness is a consequence of the following lemma.

\begin{lemma}
$F$ is a $C^1$ diffeomorphism from $\R_+^*\times (-\bar R,\bar R)$ onto $S$.
\end{lemma}

\begin{proof}
We just have to check that this reciprocal map is smooth. Let us compute the
differential of $F$. We have
$$
\partial_tF(t,\rho)=(U(\rho),U'(\rho))
\quad\textrm{ and }\quad
\partial_\rho F(t,\rho)=(tU'(\rho),tU''(\rho)).
$$
So the differential $DF$ of $F$ has not rank $2$ if and only if $(U''U-U'^2)(\rho)=0$ for
some $\rho$. By Lemma~\ref{lem:concav}, this never occurs on $(0,\bar R)$ for our choice
of $\eps$.

We have then proved that $DF$ is invertible on $\R_+^*\times (-\bar R , \bar R)$. This
implies that $F^{-1}$ is smooth and finishes the proof.
\end{proof}

The above lemma implies that the maps $t(q,w,a)=T(a,-|w|)$ and $p(q,w,a)=G(q,w,a)$ are
$C^1$ when $w\neq 0$. Concerning the behaviour near $w=0$, we use that, near $y=0$, it holds
$$
T(x,y)=x+o(y),\quad R(x,y)=-\frac2{\lambda x}y+o(y)\quad\textrm{ and
}\quad\frac{\partial R}{\partial y}(x,y)=-\frac2{\lambda x}+o(1).
$$
Hence, near $w=0$, $T(a,-|w|)=a+O(|w|^2)$ and $T(a,-|w|)$ is $C^1$ at $w=0$. Clearly
$\frac{R(a,-|w|)}{|w|}w=o(1)$ so it is continuous at $w=0$. We also have
\begin{align*}
D_w\big(\frac{R(a,-|w|)}{|w|}w\big)(h)&=-\frac{\partial R}{\partial y}(a,-|w|)\lan\frac
w{|w|^2},h\ran w+R(a,-|w|)\big(\lan-\frac w{|w|^3},h\ran w+\frac h{|w|}\big)\\
&=\frac2{\lambda a}\lan\frac w{|w|^2},h\ran w -\frac2{\lambda a}|w|\lan\frac
w{|w|^3},h\ran w+\frac2{\lambda a}|w|\frac h{|w|}+|h|o(1)\\
&=\frac2{\lambda a}h+|h|o(1).
\end{align*}
So $(a,w)\mapsto\frac{R(a,-|w|)}{|w|}w$ is $C^1$ at $w=0$ (the differential with respect
to $a$ is easier to deal with). Finally $t$ and $p$ are $C^1$ with respect to
$(q,w,a)\in\widetilde N$ and this finishes the proof of Proposition~\ref{prop:reparam}.

Let us now give the proof of Lemma~\ref{lem:concav}.

\begin{proof}[Proof of Lemma~\ref{lem:concav}]
It is enough to prove that $U''U-U'^2\neq 0$ in $(0,R_\lambda)$ since
$(U''U-U'^2)(0)=-\frac1{2\lambda}$ and $(U''U-U'^2)(R_\lambda)=-\alpha_\lambda^2$. Let us
consider $\bar \rho\in(0,R_\lambda)$.

Let $p=(0,0,1)\in\s^2\subset \R^3$, $Y=(1,0,0)\in T_p\s^2$ and $\widetilde Y$ the Killing vectorfield 
$$\widetilde Y(q)=-\meta{q}{Y}p-\meta{p}{q}Y\in T_q\s^2 .$$ 

We consider the polar parametrization of $\s^2$ given by
$G(\rho,\theta)=(\sin\rho\cos\theta,\sin\rho\sin\theta,\cos\rho)$. So writing
$q=G(\rho,\theta)$, we get 
$$
\widetilde Y(q)=-\sin\rho\cos\theta\begin{pmatrix}0\\0\\1\end{pmatrix}+\cos\rho
\begin{pmatrix}1\\0\\0\end{pmatrix}.
$$
Using $\frac\partial{\partial\rho}=\begin{pmatrix}\cos\rho\cos\theta\\ 
\cos\rho\sin\theta\\ -\sin\rho\end{pmatrix}$ and $\frac\partial{\partial\theta}=
\begin{pmatrix}-\sin\rho\sin\theta\\ \sin\rho\cos\theta\\ 0\end{pmatrix}$,
we get 
\begin{equation}\label{Ytilde}
\widetilde Y=\displaystyle\cos\theta\frac\partial{\partial\rho}-
\cot\rho\sin\theta\frac\partial{\partial\theta}
 \text{ if } \rho\neq 0 .
\end{equation}

Let $\bar q$ be the point with polar coordinates $\rho=\bar \rho$ and
$\theta=0$. Let $v$ be the function defined by 
$$
v=\frac{\meta{\widetilde Y(\bar q)}{ \nabla u_p (\bar q)}}{u_p(\bar q)}
u_p-\meta{\widetilde Y }{ \nabla u_p} \text{ in } D_p .
$$ 

On the one hand, since $\tilde Y (\rho , -\theta ) = \tilde Y (\rho , \theta)$ by
\eqref{Ytilde}, $v$ is invariant by the map $\theta\to-\theta$. On the other hand, since
$\widetilde Y$ is Killing, we have $\Delta v+\lambda v=0$. Moreover, observe that $v(\bar
q)=0$.

At $\bar q$, $\meta{\widetilde Y (\bar q)}{\nabla u_p(\bar q)}=U'(\bar \rho)$. So, along
the boundary of the disk $D$ of center $p$ and radius $\bar \rho$, we have $v(\bar
\rho,\theta)=(1-\cos\theta) U'(\bar \rho)$ and $v$ has the sign of
$U'(\bar\rho)$ on $\partial D$. If $v$ change sign inside $D$, there is then a domain
$\Ome\subset D$ such that $v$ has constant sign in $\Ome$ and vanishes on $\partial\Ome$.
Since $\Delta v+\lambda v=0$, $\lambda$ is then the first eigenvalue of $-\Delta$ in
$\Ome$. This is impossible since $\lambda$ is the first eigenvalue of $-\Delta$ in
$D_p\Supset \Ome$. 

Since $v(\bar q)=0$, the boundary maximum principle implies $v\equiv 0$ in $D$ or
$\frac{\partial}{\partial \vec \eta}v(\bar q)\neq 0$. The first case would imply
$0=v(p)=\frac{U'(\bar \rho)}{U(\bar\rho)}$ so $U'(\bar \rho)=0$ and then $0\equiv
v=-\cos\theta U'(\rho)$, \textit{i.e.} $U'(\rho)=0$, for any $\rho\in[0,\bar\rho]$ which
is impossible. Therefore, we have $0\neq\frac{\partial}{\partial \vec\eta}v(\bar
q)=\frac{U'^2(\bar \rho)}{U(\bar \rho)}-U''(\bar \rho)$, that is, $(U''U-U'^2)(\bar
\rho)\neq 0$ and the proof is finished.
\end{proof}


\subsection{Example 2}

In this subsection, we are going to generalize the preceding case. As above, we are
looking for rotationally symmetric solutions to \eqref{eq:oep2}. So the problem reduces to
the study of the ODE
\begin{equation}\label{eq:ode2}
U''+(\cot\rho)U'+f(U)=0.
\end{equation}
We will assume that $f$ has some particular property (see hypothesis~\eqref{eq:hypo} and
Proposition~\ref{prop:candfam} below) but let us begin by a general study.


\subsubsection{Solutions to \eqref{eq:ode2}}

We are interested in understanding even solutions to the ODE~\eqref{eq:ode2} in
$(-\pi,\pi)$ when the value $U(0)$ is prescribed. Let us notice that a function $U$ is
called a solution to \eqref{eq:ode2} if $U$ is $C^2$ and solves \eqref{eq:ode2} on
$(-\pi,\pi)\setminus\{0\}$.

For $0<\eps<\pi$, we denote $C_e^k(\eps)$ the space of even $C^k$ functions on
$[-\eps,\eps]$ endowed with the $C^k$ norm $\|\cdot\|_{k,\eps}$. We denote
$C_e^0(\eps)=C_e(\eps)$ and
$\|\cdot\|_{0,\eps}=\|\cdot\|_\eps$. For $g\in C_e(\eps)$, we consider first the ODE:
\begin{equation}\label{eq:odeg}
U''+(\cot \rho)U'+g=0
\end{equation}
with $U(0)=0$. This equation is equivalent to $((\sin\rho)U')'+(\sin\rho) g=0$ so the
solution must be $A(g)$ defined by
$$
A(g)(\rho)=\begin{cases}-\int_0^\rho\frac1{\sin s}\int_0^s(\sin x)g(x)dxds&\text{if }\rho\neq 0\\[3mm]
0&\text{if } \rho=0\end{cases}
$$
\begin{lemma}\label{lem:solode2}
The map $A$ has the following properties.
\begin{itemize}
\item[$(i)$] $A: C_e(\eps)\to C_e^2(\eps)$ is linear and continuous.
\item[$(ii)$] $A(g)$ solves $U''+(\cot \rho)U'+g=0$ on $[\eps,\eps]$.
\item[$(iii)$] If $g\in C_e^1(\eps)$, then $A(g)\in C_e^3(\eps)$.
\end{itemize}
\end{lemma}

\begin{proof}
Let us notice that $A(g)$ is clearly $C^2$ on $[-\eps,\eps]\setminus\{0\}$. Moreover we have
\begin{equation}\label{eq:estim}
\begin{split}
|A(g)(\rho)| & \le\int_0^{|\rho|}\frac1{\sin s}\int_0^s(\sin x)\|g\|_{|\rho|}dxds\\
 & \le\|g\|_{|\rho|}\int_0^{|\rho|}\frac{\sin\frac s2}{\cos\frac s2}ds 
  \le 2|\ln \cos\frac {|\rho|}2| \|g\|_{|\rho|}.
\end{split}
\end{equation}
This proves that $A(g)$ is continuous at $0$ and $A$ is a continuous linear map with image in $C_e(\eps)$.

We have 
$$A(g)'(\rho)=\frac{-1}{\sin \rho}\int_0^\rho(\sin x)g(x)dx $$and hence
$$
|A(g)'(\rho)|\le\frac1{\sin |\rho|}\int_0^{|\rho|}(\sin x)\|g\|_{|\rho|}dx
\le\tan\frac {|\rho|}2\|g\|_{|\rho|},
$$so $A(g)$ is $C^1$ with $A(g)'(0)=0$ and $A$ is continuous with image in $C_e^1(\eps)$. 

We also have
$$A(g)''(\rho)=\frac{\cos\rho}{\sin^2\rho}\int_0^{\rho}(\sin x)g(x)dx-g(\rho) ,$$
therefore $A(g)$ solves \eqref{eq:odeg}. Besides it holds
\begin{align*}
|A(g)''(\rho)+(1-\frac{\cos \rho}{2\cos^2\frac\rho 2})g(0)| &
\le\frac{\cos\rho}{\sin^2\rho}\int_0^{|\rho|}(\sin 
x)\|g-g(0)\|_{|\rho|}dx+\|g-g(0)\|_{|\rho|} \\
& \le(1+\frac{\cos \rho}{2\cos^2\frac\rho 2})\|g-g(0)\|_{|\rho|},
\end{align*}
thus $A(g)$ is $C^2$ with $A(g)''(0)=-\frac12g(0)$ and $A$ is continuous with image in $C_e^2(\eps)$. 

If $g$ is $C^1$ we have 
$$A(g)'''(\rho)=\frac{-1}{\sin \rho}\int_0^\rho(\sin x)g(x)dx-\frac{2\cot^2
\rho}{\sin\rho}\int_0^\rho(\sin x)g(x)dx+\cot \rho g(\rho)-g'(\rho) ,$$
so, integrating by part and $g'(0)=0$, we obtain
\begin{equation*}
\begin{split}
\left|\cot \rho g(\rho)-\frac{2\cot^2 \rho}{\sin\rho}\int_0^\rho(\sin x)g(x)dx\right|&\le
\left|\cot\rho(1-\frac{2\cos\rho(1-\cos \rho)}{\sin^2 \rho})g(0)\right|+
\left|\cot\rho\int_0^\rho g'(x)dx\right|\\
&\qquad\qquad+\left|\frac{2\cot^2\rho}{\sin\rho}\int_0^\rho(1-\cos x)g'(x)dx\right|\\
&\le\frac{\cos\rho\sin\frac{|\rho|}2}{2\cos^3\frac\rho2}|g(0)|+\rho\cot\rho\|g'\|_{|\rho|} \\
 & \qquad + \left|\frac{2\cot^2\rho(\rho-\sin\rho)}{\sin\rho}\right|\|g'\|_{|\rho|}.
\end{split}
\end{equation*}
This implies
$$
|A(g)'''(\rho)|\le(\tan\frac{|\rho|}2+\frac{\cos\rho\sin\frac{|\rho|}2}{2\cos^3\frac\rho2})\|g\|_{|\rho|}
+ \Big(1+\rho\cot\rho+\left|\frac{2\cot^2\rho(\rho-\sin\rho)}{\sin\rho}\right|\Big)\|g'\|_{|\rho|},
$$
then $A(g)$ is $C^3$ and $A$ is continuous as a map from $C_e^1(\eps)$ to $C_e^3(\eps)$.
Thus $(i)$, $(ii)$ and $(iii)$ are proved.
\end{proof}

Let us now construct a solution to \eqref{eq:ode2}.

\begin{lemma}\label{lem:solode}
Let us assume that $f$ is $C^1$. For any $t\in\R$, there is a unique even solution $U_t$
to \eqref{eq:ode2} on $(-\pi,\pi)$ such that $U_t(0)=t$. Moreover the solution is $C^1$ in
$t$.
\end{lemma}

\begin{proof}
Let us notice that the only problem is near $0$. So if we solve the problem on
$(-\eps,\eps)$, then the solution extends to $(-\pi,\pi)$ without any problem. 

If $U$ is an even solution to \eqref{eq:ode2}, then $((\sin\rho)U')'+(\sin\rho) f(U)=0$ and then
$$
U(\rho)-U(0)=-\int_0^\rho\frac1{\sin s}\int_0^s(\sin x)f(U(x))dxds=A(f\circ U)(\rho)\,.
$$

Let us then define, for $(v,t)\in C_e(\eps)\times\R$, the function $\mathcal A(v,t)$ in
$[-\eps,\eps]$ by $\mathcal A (v,t)=A(f(t+v))$. Then $U$ is a solution of \eqref{eq:ode2}
if $U-U(0)=\mathcal A (U-U(0),U(0))$. Observe that $\mathcal A(v,t)$ is continuous in
$[-\eps,\eps]$ and, by \eqref{eq:estim}, 
$$\|\mathcal A (v,t)\|_\eps\le 2 |\ln(\cos\frac\eps 2)|\sup_{[t-\|v\|_\eps,t+\|v\|_\eps]}|f| .$$
So $\mathcal A$ is a map from $C_e(\eps)\times\R$ to $C_e(\eps)$. Moreover for any $t_0\in
\R$, there is $\eps>0$ such that the ball $B_1\subset C_e(\eps)$ of radius $1$ and center
$0$ is stable by $\mathcal A (\cdot,t)$ for any $t$ close to $t_0$. By the above
description, solving \eqref{eq:ode2} with initial data $t$ consists then in finding fixed
points of $\mathcal A(\cdot,t)$.

Let us notice that the map $\mathcal A$ is $C^1$ and 
$$
D\mathcal A_{|(v,t)}(y,h)(\rho)=-\int_0^\rho\frac1{\sin s}\int_0^s(\sin x)f'(t+v(x))(y+h(x))dxds,
$$this implies that
$$
\|D\mathcal A_{|(v,t)|}\|\le 2|\ln(\cos\frac\eps2)|\sup_{[t-\|v\|_\eps,t+\|v\|_\eps]}|f'|.
$$
So, there is $\eps$ such that $B_1$ is stable by $\mathcal A(\cdot,t)$ and
$\mathcal A(\cdot,t)$ is a contraction on $B_1$. So, $\mathcal A(\cdot,t)$ has a unique
fixed point $v_t$ in $B_1$. \textit{A priori}, $t+v_t$ is just a weak solution to
\eqref{eq:ode2} ($t+v_t$ is just continuous). But Lemma~\ref{lem:solode2} tells that
$t+v_t$ is $C^2$ and a true solution of \eqref{eq:ode2}. Thus $t+v_t$ extends to
$(-\pi,\pi)$ as a solution of \eqref{eq:ode2} with initial value $t$. Moreover, since any
continuous function lies in $B_1$ if $\eps$ is sufficiently small, the function $t+v_t$ is
the unique solution to \eqref{eq:ode2} in $(-\pi,\pi)$ with initial value $t$. 

In order to prove that $v_t$ depends in a $C^1$ way of $t$, we apply the Implicit Function Theorem to the equation $\mathcal A(v_t,t)-v_t=0$. Let $t_0\in \R$ and $\eps>0$ such that $D_v \mathcal A_{|(v_{t_0},t_0)}$ is a contraction on $C_e(\eps)$. Then $D_v \mathcal A_{|(v_{t_0},t_0)}-\mathrm{id}$ is invertible from $C_e(\eps)$ to $C_e(\eps)$ and the Implicit Function Theorem applies.
\end{proof}

Let us remark that if $f$ is $C^1$ then item $(iii)$ in Lemma~\ref{lem:solode2} implies
that the solution constructed by the above lemma is $C^3$.


\subsubsection{The study of the solutions $U_t$}

We recall that we will use the solutions $U_t$ to produce rotationally symmetric
solutions to \eqref{eq:oep2}. If $u$ is a solution to \eqref{eq:oep2},
then, at the maximum point $\bar p$, the maximum principle implies $f(u(\bar p))>0$. So,
some positiveness assumption could be interesting on $f$. The second thing is that we want
to be sure that a solution $U_t$ vanishes somewhere in $(0,\pi)$ if $t>0$. We then have
the following lemma.

\begin{lemma}
Assume that $f$ is $C^1$ and positive on $\R_+^*$, then the solution $U_t$ given by
Lemma~\ref{lem:solode} with positive initial value $t$ vanishes at some $r_t\in(0,\pi)$,
we denote by $r_t$ the smallest positive zero of $U_t$. Moreover, if $r_t>\pi/2$,
$U_t$ is concave on $[\pi/2,r_t]$.
\end{lemma}

\begin{proof}
If $U_t$ never vanishes in $(0,\pi)$, then
$$
U_t'(\rho)=-\frac1{\sin\rho} \int_0^\rho\sin(x)f(U_t(x))dx<0\, ,
$$which implies that $U_t$ is decreasing and positive, and hence, it has a limit. But, for
$\rho $ close to $\pi$, it holds
$$U_t'(\rho)\sim_{\pi}-\frac1{\pi-\rho}\int_0^\pi\sin(x)f(U_t(x))dx ,$$which is not
integrable, a contradiction. Therefore, $r_t$ is well defined.

To finish the proof, observe that if $\rho\in[\pi/2,r_t]$ then $U_t''(\rho)=-\cot \rho
U_t'(\rho)-f(U_t(\rho))\le 0$, so $U_t$ is concave, as desired.
\end{proof}

\begin{remark}
Let us point out that since $U_t'(r_t)<0$, $r_t$ depends in a $C^1$ way of $t$.
\end{remark}

So let us assume that $f>0$ on $\R_+^*$. For any $p\in\s^2$ and $t>0$, we denote by
$D_{t,p}$ the disk of center $p$ and radius $r_t$. In $D_{t,p}$, the function $U_t$
define a radial function $u_{t,p}$ which is a solution to the OEP~\eqref{eq:oep2} in
$D_{t,p}$. Our aim is to prove that using some extra hypotheses on $f$ we can construct a
family of candidate solution from $(u_{t,p})_{(t,p)\in\R_+^*\times\s^2}$.

As in the case of \eqref{eq:oep3}, we need to go further in the description of the
solutions of \eqref{eq:ode2}. So for the rest of this section, we fix some extra
assumptions on $f$, the assumption is 

\begin{equation*}\tag{$H$}\label{eq:hypo}
\textrm{$f$ is a $C^1$ function so that  $f(x)>0$ and $f(x)\ge xf'(x)$ for any $x\in\R_+^*$.}
\end{equation*}

Since the family of functions $U_t$ is $C^1$ in $t$, we define
$H_t=\frac{\partial}{\partial t}U_t$. Observe that $H_t$ is a solution to 
\begin{equation}\label{eq:odeh}
H_t''+(\cot \rho)H_t'+f'(U_t)H_t=0,
\end{equation}
with $H_t(0)=1$. We use the notation $h_{t,p}$ to denote the rotationally symmetric
function in $D_{t,p}$ associated to $H_t$.

\begin{lemma}
Let $f$ be a function satisfying \eqref{eq:hypo}, then for any $t$, $H_t$ is positive in $(-r_t,r_t)$.
\end{lemma}

\begin{proof}
Since $H_t(0)=1$, if the lemma were not true, there is $\bar \rho\in(0,r_t)$ such that
$H_t(\bar\rho)=0$ and $H_t(\rho)>0$ for $\rho\in(-\bar \rho,\bar \rho)$. Given $p\in\s^2$,
let $D_{\bar \rho}$ be the geodesic disk of center $p$ and radius $\bar\rho$. The function
$h_{t,p}$ is positive in $D_{\bar \rho}$ and vanishes on $\partial D_{\bar\rho}$.
$u_{p,t}$ is positive in $\overline D_{\bar\rho}$ so
$0<\delta=\min_{D_{\bar\rho}}\frac{u_{t,p}}{h_{t,p}}$ is well defined. Then the function
$w=u_{t,p}-\delta h_{t,p}$ is non negative in $D_{t,p}$ and vanishes at some point inside
it. By Hypothesis~\eqref{eq:hypo}, we have
$$
\Delta w+f'(u_{t,p})w=f'(u_{t,p})u_{t,p}-f(u_{t,p})\le 0.
$$
By the maximum principle, $w\equiv 0$ which is impossible since $w>0$ on $\partial
D_{\bar\rho}$.
\end{proof}

From this lemma we see that $U_t$ is increasing in $t$. We are now interested in the
behavior when $t$ goes to $+\infty$.

\begin{lemma}\label{lem:proper}
Let $f$ be a function satisfying \eqref{eq:hypo}. Let $t_n$ be a sequence going to
$+\infty$ and $\rho_n\in[0,r_{t_n}]$ then
$\|(U_{t_n}(\rho_n),U_{t_n}'(\rho_n))\|\to+\infty$. Actually, either $U_{t_n}(\rho_n)\to
+\infty$ or $U_{t_n}'(\rho_n)\to-\infty$.
\end{lemma}

\begin{proof}
Let us assume the lemma is false. For $t>0$, we denote $\mathcal
U_t(\rho)=U_t^2(\rho)+U_t'^2(\rho)$. So we can find $t_n\to+\infty$ and $\rho_n\in
[0,r_{t_n}]$ such that $(\mathcal U_{t_n}(\rho_n))_n$ is bounded. Since $r_t$ is non
decreasing, we can define
$r_\infty=\lim_{t\to\infty} r_t$.

Since $f$ satisfies \eqref{eq:hypo}, $f(x)\le f(1)x=\lambda x$ for $x\in[1,+\infty)$ and
there is $a>0$ such that $f(x)\le a+\lambda x$ on $\R_+$. If $\rho_n>\pi/2$, the concavity
of $U_{t_n}$ in $[\pi/2,\rho_n]$ implies that 
$$
0 \le U_{t_n}(\pi/2)\le U_{t_n}(\rho_n)+U_{t_n}'(\rho_n)(\pi/2-\rho_n) \text{ and }
U_{t_n}'(\rho_n)\le U_{t_n}'(\pi/2)\le 0 .
$$

Since $\mathcal U_{t_n}(\rho_n)$ is bounded, this implies $\mathcal U_{t_n}(\pi/2)$ is
bounded and we can always assume that $\rho_n\le \pi/2$. Using $0\le f(x)\le a+\lambda x$
on $\R_+$, for $0<\rho<\min(r_t,\pi/2)$, we have 
$$
-(1+\cot \rho)\mathcal U_t\le \mathcal U_t'=2U_tU_t'-((\cot\rho)U_t'-f(U_t))U_t'\le
a\sqrt{\mathcal U_t}+\lambda \mathcal U_t ,
$$
thus, since $\mathcal U_{t_n}(\rho_n)$ stays bounded, we have $(\mathcal U_{t_n}(\rho))_n$
is bounded for any $\rho\in (0,\min(r_\infty,\pi/2))$.

Let us now prove that $U_{t_n}(\rho)$ goes to $\infty$ for $\rho$ close to $0$ which is a
contradiction with the boundedness of $\mathcal U$. 

There is $d_0\in(0,r_\infty)$ such that $U_{2}(\rho)\ge 1$ for $|\rho|\le d_0$. From the
monotonicity $t\mapsto U_t$, we thus have $U_t(\rho)\ge 1$ for $|\rho|\le d_0$ and $t\ge
2$. Let $\mu>0$ be the first eigenvalue of $-\Delta$ in a geodesic disk of $\s^2$ of
radius $d_0$. Let $\tilde\mu\ge \mu$ such that $a+\lambda x< \tilde\mu x$ if $x\ge 1$. Let
$d_1\le d_0$ be the radius of a geodesic disk in $\s^2$ whose first eigenvalue is
$\tilde\mu$. 

The first eigenfunction in a disk of radius $d_1$ is a rotationally symmetric function
generated by a profile curve $W$ solving in $[-d_1,d_1]$ the equation
$$
W''+(\cot\rho)W'+\tilde \mu W=0
$$
with $W(0)=1$. Let us prove that for $t\ge 2$, $U_t\ge tW$ on $[-d_1,d_1]$. Notice that
$U_t\ge 1$ on $[-d_1,d_1]$, so $U_t> tW$ close to $-d_1$ and $d_1$. Moreover, at $\rho=0$, it holds
$$
U_t(0)=t=tW(0) \text{ and } U_t''(0)=-\frac12f(t)> -\frac{\tilde\mu}2t=tW''(0) .
$$
So $U_t-tW$ vanishes at $0$ and is positive near $0$. If $U_t-tW\ge 0$ were not true,
there is $d_2\in(0,d_1)$ such that $U_t(d_2)-tW(d_2)=0$. Let
$\delta=\sup_{[-d_2,d_2]}U_t/W$. Since $U_t-tW$ is positive near $0$, $\delta>t$. We have
$U_t-\delta W\le 0$ on $[-d_2,d_2]$, $U_t(d_2)-\delta W(d_2)< 0$ and there is some $d_3$
such that $U_t(d_3)-\delta W(d_3)=0$. At that value $d_3$ we obtain
\begin{align*}
0\ge (U_t-\delta W)''(d_3)&=-(\cot d_3)U_t'(d_3)-f(U_t)(d_3)+(\cot d_3)\delta
W'(d_3)+\tilde \mu \delta W(d_3)\\
&>-(\cot d_3)(U-\delta W)'(d_3)-\tilde\mu(U-\delta W)(d_3)=0 ,
\end{align*}
which gives us a contradiction and $U_t\ge tW$ on $[-d_1,d_1]$ as claimed.

Using this last inequality, we obtain $\lim U_{t_n}(\rho)=+\infty$ for any
$\rho\in(-d_1,d_1)$. This contradicts $\mathcal U_{t_n}(\rho)$ stays bounded as $n\to\infty$. The
proof is finished.
\end{proof}

We need one more lemma concerning $U_t$ and $H_t$ which is similar to Lemma~\ref{lem:concav}.

\begin{lemma}\label{lem:jacob}
Let $f$ be a function satisfying \eqref{eq:hypo}. If $t>0$ then $H_tU_t''-U_t'H_t'\neq 0$ on $[-r_t,r_t]$.
\end{lemma}

\begin{proof}
First we see that $(H_tU_t''-U_t'H_t')(0)=U_t''(0)=-\frac12f(t)<0$. Consider $\bar
\rho\in(0,r_t]$. If $H_t(\bar \rho)=0$, then $H_t'(\bar \rho)\neq 0$ since $H_t$ is a non
vanishing solution of \eqref{eq:odeh} and $(H_tU_t''-U_t'H_t')(\bar \rho)=-(U_t'H_t')(\bar \rho)\neq
0$. So we can assume $H_t(\bar\rho)\neq 0$.

As in the proof of Lemma~\ref{lem:concav}, let $p=(0,0,1)\in \s^2$, $Y=(1,0,0)\in T_p\s^2$
and $\widetilde Y$ the Killing vectorfield $\widetilde Y=-\meta{q}{Y}p-\meta{p}{q}Y$. We
also consider the polar coordinates $(\rho,\theta)$ around $p$ and $\bar q$ the point with
polar coordinates $\rho=\bar\rho$ and $\theta=0$. Let $v$ be the function defined by 
$$
v=\frac{\meta{\widetilde Y(\bar q)}{ \nabla u_{t,p} (\bar q)}}{h_{t,p}(\bar q)}
h_{t,p}-\meta{\widetilde Y }{ \nabla u_{t,p}} \text{ in } D_p ,
$$
such that $v(\bar q)=0$ (recall that $h_{t,p}(\bar q)\neq 0 $ since $H_t(\bar \rho)\neq
0$). Since $\widetilde Y$ is Killing, we have $\Delta v+f'(u_{t,p})v=0$. 

Using polar coordinates, we have $v(\rho,\theta)=\frac{U_t'(\bar \rho)}{H_t(\bar
\rho)}H_t(\rho)-\cos\theta U_t'(\rho)$. Let $D_{\bar \rho}$ denote the disk of radius
$\bar \rho$ and center $p$. Since $U_t'(\bar \rho)<0$, the maximum of $v$ on $\partial
D_{\bar \rho}$ is reached at $\theta=0$ and $v\le 0$ on $\partial D_{\bar \rho}$.

Let us prove $v\le 0$ in $D_{\bar\rho}$. If this were not the case, we consider
$\Ome=\{q\in D_{\bar \rho} \, | \, \, v(q)>0\}$. Since $h_{t,p}>0$ in $D_{\bar \rho}$,
$0<\delta=\min_\Ome \frac {h_{t,p}} v$ is well defined and satisfies $h_{t,p}-\delta v\ge
0$ in $\Ome$ and $h_{t,p}-\delta v$ vanishes at some point in $\Ome$. Since
$h_{t,p}-\delta v$ solves $\Delta u+f'(u_t)u=0$, the maximum principle gives
$h_{t,p}-\delta v\equiv 0$ in $\Ome$ but $h_{t,p}-\delta v>0$ on $\partial \Ome$. We get
our contradiction and $v\le0$ in $D_{\bar\rho}$. 

By the boundary maximum principle this implies either $v\equiv 0$ in $D_{\bar \rho}$ or
$\frac{\partial}{\partial \vec\eta}v(\bar q)>0$. The first case is impossible since
$v(p)=\frac{U_t'(\bar \rho)}{H_t(\bar\rho)}\neq 0$. The second case implies
$0<\frac{\partial}{\partial \vec\eta}v(\bar q)=\frac{U_t'(\bar \rho)}{H_t(\bar\rho)}H_t'(\bar
\rho)-U_t''(\bar\rho)$ and the lemma is proved.
\end{proof}


\subsubsection{The family of candidate solutions}

Now we are interested in finding a good reparametrization of
$(u_{t,p})_{(t,p)\in\R_+^*\times\s^2}$. More precisely we are going to prove the main
result of this section.
\begin{proposition}\label{prop:candfam}
Let us assume that $f$ satisfies \eqref{eq:hypo}. Then \eqref{eq:oep2} admits a family of
candidate solutions.
\end{proposition}

Let us consider
$$
F:\begin{matrix}
\R\times(-\pi,\pi)&\to &\R^2\\
(t,\rho)&\mapsto&(U_t(\rho),U_t'(\rho)).
\end{matrix}
$$
Let us also define $\Delta=\{(t,\rho)\in\R_+^*\times(-\pi,\pi) \, |\, \, -r_t<\rho<r_t\}$
and $\overline\Delta$ be the closure of $\Delta$ in $\R_+^*\times(\pi,\pi)$.

\begin{lemma}
The map $F$ is a diffeomorphism from $\Delta$ to $\R_+^*\times\R$.
\end{lemma}

\begin{proof}
First we notice that $F(\Delta)\subset \R_+^*\times\R$. We have 
$$
\partial_tF(t,\rho)=(H_t(\rho),H_t'(\rho))
\quad\textrm{ and }\quad
\partial_tF(t,\rho)=(U_t'(\rho),U_t''(\rho)).
$$
So the differential $DF$ of $F$ has not rank $2$ if and only if
$(H_tU_t''-H_t'U_t')(\rho)=0$ for some $\rho\in(-r_t,r_t)$. By
Lemma~\ref{lem:jacob}, this never occurs so $F$ is a local diffeomorphism.

Let us prove that $F$ is proper. Let $(t_n,\rho_n)_n$ be a proper sequence in $\Delta$. If
$t_n\to +\infty$, Lemma~\ref{lem:proper} tells that $\|F(t_n,\rho_n)\|\to +\infty$. If
$t_n\to 0$, by monotonicity, $U_{t_n}(\rho_n)\le U_{t_n}(0)=t_n\to 0$ and $F(t_n,\rho_n)$
goes to $\partial (\R_+^*\times\R)$. If $t_n\to a>0$ then $\rho_n\to r_a$ and
$U_{t_n}(\rho_n)\to U_a(r_a)=0$. Therefore, in any case, $F(t_n,\rho_n)$ goes to $\partial
(\R_+^*\times\R)$. Thus $F$ is proper which implies that $F:\Delta\to \R_+^*\times\R$ is a
covering map. Since $\R_+^*\times\R$ is simply connected $F$ is a diffeomorphism.
\end{proof}

We have $F(t,r_t)=(0,U_t'(r_t))$ and $\partial_t U_t'(r_t)= H_t'(r_t)+
U_t''(r_t)\partial_t r_t= H_t'(r_t)- \frac{H_t(r_t)}{U_t'(r_t)}
U_t''(r_t) <0$ by Lemma~\ref{lem:jacob}. This implies that $F$ extend to a bijective map
on $\overline\Delta$, notice that $F(t,-r_t)=(0,-U_t'(r_t))$. Actually we have much more.

\begin{lemma}
There is an open neighborhood $\widetilde\Delta$ of $\overline\Delta$ in
$\R_+^*\times(-\pi,\pi)$ such that $F$ is a diffeomorphism from $\widetilde\Delta$ onto
its image.
\end{lemma}

\begin{proof}
Let $P$ be the projection map 
$$
P:\begin{matrix}
\R^2&\to &\{0\}\times\R\\
(x,y)&\mapsto&(0,y).
\end{matrix}
$$

Let $p\in \partial \Delta$, since $DF(p)$ has rank $2$ there is a neighborhood $\mathcal
O_p$ of $p$ such that $F$ is a diffeomorphism from $\mathcal O_p$ onto its image. Moreover
by reducing $\mathcal O_p$ if necessary, we can assume $F(\mathcal O_p\setminus
\overline\Delta)\subset \R_-^*\times\R$. 

Actually, by reducing $\mathcal O_p$ once more, we can assume that for any $q\in \mathcal
O_p$, the segment $[F(q),P(F(q))]$ is contained in $F(\mathcal O_p)$. We claim that
$\widetilde \Delta=\Delta\cup(\bigcup_{p\in\partial \Delta} \mathcal O_p)$ satisfies the
expected property.

Since $F$ is a local diffeomorphism on $\widetilde \Delta$, we only have to check the
injectivity. Since $F(\mathcal O_p\setminus \overline\Delta)\subset \R_-^*\times\R$ and
$F$ is injective in $\overline \Delta$, if $F$ were not injective, there is $q_i\in
\mathcal O_{p_i}$ ($i=1,2$) such that $F(q_1)=F(q_2)$. Let $\gamma : [0,1] \to \R^2$ be
the segment $[F(q_i),P(F(q_i))]$, with $\gamma (0) = F(q_i)$ and $\gamma (1) = P(F(q_i))$,
and $\gamma_i=F^{-1}(\gamma)$ in $\mathcal O_{p_i}$. We have $F(\gamma_i(1))=P(F(q_i))$.
So since $F$ is injective on $\partial\Delta$, $\gamma_1(1)=\gamma_2(1)$. Since
$F(\gamma_1(s))=F(\gamma_2(s))$ and $F$ is a local diffeomorphism, $\gamma_1=\gamma_2$ and
$q_1=\gamma_1(0)=\gamma_2(0)=q_2$: $F$ is injective. 
\end{proof}

Now the end of the proof of Proposition~\ref{prop:candfam} is similar to the one of the
\eqref{eq:oep3} case. More precisely, let us define $S=F(\widetilde \Delta)$, $\widetilde
N=\{((q,w),a)\in T\s^2\times\R \, | \,\, (a,|w|)\in S\}$ and 
$$F^{-1}(x,y)=(T(x,y),R(x,y)) \text{ for } (x,y)\in S .$$ 
Then the computations are the same except that the estimates of $R$ become
$$
R(x,y)=-\frac2{f(x)}y+o(y)\quad\textrm{ and
}\quad\frac{\partial R}{\partial y}(x,y)=-\frac2{f(x)}+o(1).
$$


\section{Main Theorem and applications}\label{sec:mainth}

Before stating our main theorem, let us give a last definition. Let $(M,g)$ be a
Riemannian surface and $Q$ be a quadratic form on $M$. A curve
$\gamma$ on $M$ is a \textit{line of curvature} of $Q$ if $\gamma'(t)$ is an eigenvector
of $Q_{\gamma(t)}$ for any $t$.

Now, with the definition of a family of candidate solutions associated to an OEP, we can announce:

\begin{theorema}
Let $(M,g)$ be a complete connected Riemannian surface and $f:\R \to \R $ be a Lipschitz
function. Let $\Omega \subset M$ be a bounded connected domain with a $C^2$
boundary $\partial \Omega$ and $u \in C^3 (\overline\Omega)$ be a solution to
the OEP 
$$
\begin{cases}
\Delta{u}+f(u)=0  &\text{ in }   \Omega,\\
u>0  				 &\text{ in }   \Omega,\\
u=0                     &\text{ on }\partial\Omega,\\
\lan\nabla{u},\vec{\eta}\ran=\alpha &\text{ on } \partial\Omega .
\end{cases}
$$

Assume that the above OEP admits a family of candidate solutions $\mathcal C$. Then, there
exists a $C^1$ traceless symmetric bilinear form $Q$ on $\overline{\Omega}$ such that 
\begin{enumerate}
\item $Q$ vanishes at some $x \in \overline\Omega$ if, and only if, $u$ has a contact of
order $k \geq 2$ at $x$ with some $\bar v \in \mathcal C $.
\item $Q$ vanishes identically on $\overline\Omega$ if, and only if, $u \in \boC _\alpha$.
\item If $Q$ does not vanish identically on $\overline\Omega$, then $Q $ has only isolated
(interior and boundary) zeroes. Moreover, the null directions of $Q$ determine on
$\overline\Omega$ two $C^1$ line fields with isolated singularities of negative index.
\item The boundary $\partial \Omega$ is a line of curvature of $Q$.
\end{enumerate}
\end{theorema}


\subsection{Applications}

As a consequence of Theorem A, we can prove some uniqueness property for solutions of some OEP.

\begin{theorem}\label{th:simplyconnected}
Let $f$ be a Lipschitz function and $\Omega \subset M$ be a compact simply-connected domain with a
$C^2$ boundary $\partial \Omega$ where the OEP \eqref{eq:oep} admits a solution
$u \in C^3 (\Omega)$. If the OEP admits a family of candidate solutions $\mathcal C$, then
$u \in \mathcal C _\alpha$.
\end{theorem}

Combining the above theorem with the candidate families constructed in Section \ref{Examples}, we obtain:

\begin{corollary}[BCN conjecture in $\s^2$]\label{cor:bcn}
Let us assume that $f$ satisfies \eqref{eq:hypo}. Let $u$ be a solution of \eqref{eq:oep}
in some topological disk $\Ome \subset \s^2$ with $C^2$ boundary. Then, $\Ome$ is a geodesic disk
(centered at some point $p\in \s^2$) and $u$ is rotationally symmetric (with respect to
the center $p$).
\end{corollary}

\begin{proof}[Proof of Theorem~\ref{th:simplyconnected}]
Let $Q$ be the traceless symmetric bilinear form given by Theorem~A. If $Q$ vanishes
identically then $u \in \mathcal C _\alpha$ and we are done. So, we assume that $Q$ does
not vanish identically and we will get a contradiction. We follow ideas already appearing
in the work of Choe~\cite{Choe}.

Let $\Phi_+$ be a diffeomorphism from the north hemisphere $\s_+=\{(x,y,z)\in\s^2 \, |\, \, z\ge 0\}$ to $\overline\Ome$. Let us define $\Phi_-=\Phi_+\circ S$ on the south hemisphere $\s_-$ where $S$ is the symmetry with respect to the equator $E$. We define $\tilde Q=\Phi_+^*Q$ on $\s_+$ and $\tilde Q=\Phi_-^*Q$ on $\s_-$. Since $\partial\Ome$ is a line of curvature of $Q$ these two definitions coincide on $E$. By Theorem~A, $\tilde Q$ is a symmetric bilinear form on $\s^2$ which has isolated
zeroes and is a Lorentzian metric outside its zeroes. Moreover the null directions of
$\tilde Q$ determine on $\s^2$ two line fields with isolated singularities of negative
index. This gives a contradiction by the Poincar\'{e}-Hopf Index Theorem.
\end{proof}

Another consequence is the classification of simply-connected harmonics domains in $\s^2$. A domain $\Omega \subset \s ^2$ with regular boundary is harmonic if the mean value of any harmonic function on $\Omega$ equals its mean value on $\partial \Omega$ and they are characterized by supporting a solution to the {\it Serrin Problem}
\begin{equation*}
\begin{cases}
\Delta{u}+1=0  &\text{ in }   \Omega,\\
u=0                     &\text{ on }\partial\Omega,\\
\langle\nabla{u},\vec{\eta}\rangle_{g}= \alpha &\text{ on } \partial\Omega .
\end{cases}
\end{equation*}

Note that in the above case, $u$ must be positive on the interior of $\Omega$ by the maximum principle. Thus, Theorem \ref{th:simplyconnected} and the candidate families constructed in Section \ref{Examples} give

\begin{corollary}[Serrin Problem in $\s^2$]\label{cor:Serrin}
Any simply connected harmonic domain in $\s ^2$ is a geodesic ball. 
\end{corollary}

The Schiffer conjecture D can be stated as (cf. \cite{Sch,Shk}): Given a bounded connected open domain $\Omega \subset \s ^2$ with a regular boundary and such that the complement of its closure is connected, the existence of a solution to 
\begin{equation*}
\begin{cases}
\Delta{u}+\lambda u=0  &\text{ in }   \Omega,\\
u=0                     &\text{ on }\partial\Omega,\\
\langle\nabla{u},\vec{\eta}\rangle_{g}= \alpha &\text{ on } \partial\Omega ,
\end{cases}
\end{equation*}implies that $\Omega$ is a geodesic ball. When $\lambda$ is the first eigenvalue of the Laplacian then $u >0$ on $\Omega$. Thus, the Schiffer conjecture D for the first eigenvalue can be stated as:

\begin{quote}
{\bf First Schiffer conjecture D in $\s^2$:} {\it If there exits a solution to
\begin{equation*}
\begin{cases}
\Delta{u}+\lambda u=0  &\text{ in }   \Omega,\\
u>0  &\text{ in }   \Omega,\\
u=0                     &\text{ on }\partial\Omega,\\
\langle\nabla{u},\vec{\eta}\rangle_{g}= \alpha &\text{ on } \partial\Omega ,
\end{cases}
\end{equation*}then $\Omega$ is a geodesic ball.}
\end{quote}

Therefore, the previous considerations lead us to
\begin{corollary}[First Schiffer Conjecture D in $\s^2$]\label{cor:Schiffer}
The first Schiffer conjecture is true in $\s ^2$.
\end{corollary}

A last application concerns the case where $f(x)=x-x^3$ which is the usual Allen-Cahn
non-linearity. Actually such $f$ does not satisfy \eqref{eq:hypo} since $f$ is negative on
$(1,\infty)$ but $f(x)\ge xf'(x)$ for $x\ge 0$. So if $u$ is solution of \eqref{eq:oep} on
$\Ome\subset\s^2$ and $M=\max u$, the maximum principle implies that $f(M)>0$ and $M<1$.
We can then consider a function $g$ which coincides with $f$ on $[0,M]$ and satisfying
\eqref{eq:hypo}. As a consequence $\Ome$ is a $g$-extremal domain. We thus obtain
\begin{corollary}
If $f(x)=x-x^3$, a $f$-extremal disk in $\s^2$ with $C^2$ boundary is a geodesic disk.
\end{corollary}


\subsection{Proof of Theorem A}

Once the quadratic form $Q$ will be defined, we will consider its $(2,0)$ part $\boP$: if
$z$ is a local conformal parameter, $\boP = P(z) dz^2$ where $P(z)=Q(\pz,\pz)$. Properties
of $Q$ can then be deduced from properties of $\boP$ since $Q = \mathcal P +
\overline{\mathcal P}$ (cf. \cite{Ho}). 

Let $\Omega \subset (M,g) $ be a bounded domain with $C^2$ boundary $\partial \Omega$. Let
$u$ be a solution to the problem~\eqref{eq:oep2}.

Let $(v_{q,w,a})_{(q,w,a)\in N}$ be a candidate family
associated to the OEP~\eqref{eq:oep}, we recall that $N=TM\times \R_+\setminus\{(q,0,0)\in TM\times
\R;\ q\in M\}$. For any $x\in \Omega$, let us define the symmetric bilinear form $Q_x : T_x M
\times T_x M \to \R $ given by 
$$
Q_x := \nabla^2 u (x) - \nabla^2 v_{x,\nabla u(x),u(x)} (x)
$$
where $\nabla^2$ is the Hessian operator. Observe that $Q$ is well-defined in $\overline\Omega$ and is
$C^1$ since $u$ is $C^3$ and $(v_{q,w,a})_{(q,w,a)\in N}$ is a smooth family. The
definition of $Q$ proves item 1 in Theorem A. 
\begin{claim}\label{cl:a}
$Q_x$ is traceless so either $Q_x = 0$ or $Q_x$ is a Lorentzian metric.
\end{claim}

\begin{proof}[Proof of Claim~\ref{cl:a}]
Let us compute the trace of $Q_x$ (w.r.t. $g$)  at each point $x \in \overline\Omega$:
$$
\Tr_g Q_x = \Delta u (x)- \Delta v_{x,\nabla u(x),u(x)} (x)) = - f(u(x))+f(v_{x,\nabla
u(x),u(x)}(x)) = -f(u(x)) +f(u(x)) = 0,
$$
where we have used item (b) of the family of candidate solutions. This finishes the proof
of Claim~\ref{cl:a}.
\end{proof}

A second point is easy to verify, it is statement~4 in Theorem A.

\begin{claim}\label{cl:b}
The boundary $\partial \Omega$ is a line of curvature of $Q$.
\end{claim}

\begin{proof}[Proof of Claim~\ref{cl:b}]
Let $\vec \tau $ the unit tangent vector field along $\partial \Omega$. Since
$\meta{\nabla u}{\vec \eta} = \alpha $ is constant along $\partial \Omega$, if we
differentiate with respect to $\vec \tau$ we get 
$$
0 = \meta{\nabla _{\vec \tau} \nabla u}{\vec \eta} + \meta{\nabla u}{\nabla _{\vec\tau}\vec\eta} 
= \nabla ^2 u (\vec\tau,\vec\eta) + \alpha \meta{\vec\eta}{\nabla _{\vec\tau}\vec\eta} 
=  \nabla ^2 u (\vec\tau,\vec\eta) ,
$$
where we have used that $\meta{\vec\eta}{\nabla _{\vec\tau}\vec\eta} =0$ (observe
we use $\partial \Omega \in C^2$). Note that the
above holds for any solution satisfying the boundary condition, as it does the candidate
family, hence
$$
Q_x (\vec\tau (x),\vec\eta(x)) = \nabla ^2 u (x) (\vec\tau (x),\vec\eta(x)) -\nabla ^2
v_{x,\nabla u(x),u(x)}(x)(\vec\tau (x),\vec\eta(x)) = 0,
$$
that is, $\partial \Omega$ is a line of curvature of $Q$ since $Q$ is trace free.
\end{proof}

Now, we study the behavior of $Q $ near a point where $Q_x = 0$. To do this we introduce $z \in \mathcal
U \subset \Omega$ a local conformal parameter in $\mathcal U$, i.e., there exists a
positive function $\lambda \in C^2 (\mathcal U)$ such that $ g = \lambda(z) |dz|^2 $. We
have a first observation.

\begin{lemma}\label{Lemma}
Given $v \in C^2 (\mathcal U)$ a solution to $\Delta v+f(v)=0$, consider the quadratic differential 
\begin{equation}\label{Quadratic}
\mathcal H _v = H _v(z) \, dz^2 = (\nabla ^2 v)(\pz ,\pz) \, dz ^2.
\end{equation}

Then, it holds
\begin{equation}\label{Eq:Hessz}
(H_v)_{\zb} = - \frac{\lambda}{4}\left(  f' (u)  +2 K_g \right)\langle\nabla  v , \pz \rangle,
\end{equation}
where  $K_g$ denote the Gaussian curvature of $g$.
\end{lemma}
\begin{proof}
First, note that
$\nabla _\pz \pzb = \nabla _{\pzb} \pz = 0$ since $z$ is a local conformal parameter for
$g$. Moreover, a straightforward computation shows
$$ \langle\nabla _{\pzb}\nabla v , \pz\rangle = \dfrac{\lambda}{4} \Delta v .$$

Then, 
\begin{equation*}
\begin{split}
(H_v)_{\zb} &= \pzb\left( (\nabla ^2 v)(\pz ,\pz) \right) 
= \pzb\left( \langle\nabla _\pz \nabla v,\pz\rangle \right) 
= \langle\nabla _{\pzb} \nabla _\pz \nabla v,\pz\rangle \\
&= \lan\nabla _{\pz} \nabla _{\pzb} \nabla  v,\pz\ran + \lan R(\pzb ,\pz ) \nabla v ,\pz\ran \\ 
& = \pz\left( \lan \nabla _{\pzb} \nabla v,\pz\ran \right) - \lan \nabla _{\pzb} \nabla
v,\nabla _{\pz}\pz\ran + \lan R(\pzb ,\pz ) \nabla v ,\pz\ran \\
&= \pz\left( \dfrac{\lambda}{4} \Delta v \right) -\frac{\lambda _z}{\lambda}\lan \nabla
_{\pzb} \nabla v,\pz\ran - \frac{\lambda K _g}{2} \lan \nabla v,\pz\ran \\
 & = \dfrac{\lambda}{4} \pz\left( \Delta v \right) - \frac{\lambda K _g}{2} \lan \nabla
 v,\pz\ran = \frac{\lambda}{4}\left( \lan\nabla \Delta v , \pz \ran -2 K_g \, \lan\nabla v , \pz \ran
 \right) ,
\end{split}
\end{equation*}
now, using that $\Delta v = - f (v)$ we get $\nabla \Delta v = - f' (v) \nabla v$, and
\eqref{Eq:Hessz} holds.
\end{proof}

Now we introduce $\boP$ the $(2,0)$ part of $Q$: \textit{i.e.} $\boP=Q_z(\pz,\pz)dz^2=P(z)
dz^2$ where $z$ is a local conformal parameter. Since $Q$ is trace free, we have
$$
Q = \boP + \overline{\boP} .
$$

\begin{claim}\label{cl:c}
The function $P$ satisfies $P_\zb(z)=\beta(z)\overline P(z)$ where $\beta$ is some
continuous complex function.
\end{claim}

\begin{proof}[Proof of Claim~\ref{cl:c}]
Let us fix some $z_0$ where we are going to compute $P_\zb$, we denote by $u_0=u(z_0)$
and $w_0=\nabla u(z_0)$. We finally define $\bar v$ to be $v_{z_0,w_0, u_0}$. $u$ and
$\bar v$ are then two solutions to
\eqref{eq:oep} which satisfy $u(z_0)=\bar v(z_0)$ and $\nabla u(z_0)=\nabla \bar
v(z_0)$. In the following computations, we use the notation $v_{q,w,a}(z)=v(z,q,w,a)$
since we will compute derivatives with respect to all these parameters (this computation
can be done because of the smoothness of a family of candidate solutions).
Hence from \eqref{Eq:Hessz}, we get
\begin{align*}
P_{\zb}(z_0)&=(H_u)_{\zb}(z_0)-(H_{\bar v})_{\zb}(z_0)-D_q\nabla^2v(z_0,z_0,w_0,u_0)(\pz,\pz)(\pzb )\\
&\quad-D_w\nabla^2v(z_0,z_0,w_0,u_0)(\pz,\pz)(\nabla_\pzb \nabla u)
- D_a\nabla^2v(z_0,z_0,w_0,u_0) (\pz,\pz )(\pzb u)\\
&=-\frac\lambda 4(f'(u_0)+2K_g)\lan w_0,\pz\ran+\frac\lambda 4(f'(u_0)+2K_g)\lan w_0,\pz\ran\\
&\quad-D_q\nabla^2v(z_0,z_0,w_0,u_0)(\pz,\pz)(\pzb )
-D_w\nabla^2v(z_0,z_0,w_0,u_0)(\pz,\pz)(\nabla_\pzb \nabla u)\\
&\quad- D_a\nabla^2v(z_0,z_0,w_0,u_0) (\pz,\pz )\meta{w_0}{\pzb}\\
&=-D_q\nabla^2v(z_0,z_0,w_0,u_0)(\pz,\pz)(\pzb )
-D_w\nabla^2v(z_0,z_0,w_0,u_0)(\pz,\pz)(\nabla_\pzb \nabla u)\\
&\quad- D_a\nabla^2v(z_0,z_0,w_0,u_0) (\pz,\pz )\meta{w_0}{\pzb}\,.\\
\end{align*}

We can do the same computation by replacing $u$ by $\bar v$ but in this case, since $\bar
v$ is a candidate solution $\bar v=v_{x,\nabla \bar v(x),\bar v(x)}$ for any $x$, the
associated $\boP ^{\bar v}$ vanishes. Let us explain this. Consider the symmetric
quadratic form 
$$
Q^{\bar v}_x := \nabla^2 \bar v (x) - \nabla^2 v_{x,\nabla \bar v(x),\bar v(x)} (x), 
$$and hence we can write 
$$
Q ^{\bar v}= \boP ^{\bar v}+ \overline{\boP^{\bar v}} ,
$$but, since $\bar v$ is a candidate solution, we have $\boP ^{\bar v} \equiv 0$. Hence, we obtain
\begin{align*}
0&=-D_q\nabla^2v(z_0,z_0,w_0,u_0)(\pz,\pz)(\pzb )-D_w\nabla^2v(z_0,z_0,w_0,u_0)(\pz,\pz)(\nabla_\pzb \nabla \bar v)\\
&\quad- D_a\nabla^2v(z_0,z_0,w_0,u_0) (\pz,\pz )\lan w_0,\pzb\ran\,.\\
\end{align*}	
So gathering the above two equations gives
$$
P_{\zb}(z_0)=-D_w\nabla^2v(z_0,z_0,w_0,u_0)(\pz,\pz)(\nabla_\pzb \nabla u-\nabla_\pzb \nabla \bar v)\,.
$$

We also have
\begin{align*}
(\nabla_\pzb \nabla u-\nabla_\pzb \nabla \bar v)(z_0)&=\frac2\lambda\Big(\lan\nabla_\pzb
\nabla u-\nabla_\pzb \nabla \bar v,\pzb\ran\pz+
\lan\nabla_\pzb \nabla u-\nabla_\pzb \nabla \bar v,\pz\ran\pzb\Big)\\
&=\frac2\lambda\Big((\nabla^2u(\pzb,\pzb)-\nabla^2\bar
v(\pzb,\pzb))\pz+(\nabla^2u(\pzb,\pz)-\nabla^2\bar v(\pzb,\pz))\pzb\Big)\\
&=\frac2\lambda\Big(\overline P(z_0)\pz+\frac\lambda4(\Delta u(z_0)-\Delta \bar v(z_0))\pzb\Big)\\
&=\frac2\lambda\Big(\overline P(z_0)\pz+\frac\lambda4(-f(u_0)+f(u_0))\pzb\Big)
=\frac2\lambda\overline P(z_0)\pz\,.
\end{align*}
Thus we obtain
$$
P_\zb(z_0)=-D_w\nabla^2v(z_0,z_0,w_0,u_0)(\pz,\pz)(\frac2\lambda\overline
P(z_0)\pz)=-\frac2\lambda D_w\nabla^2v(z_0,z_0,w_0,u_0)(\pz,\pz)(\pz)\overline P(z_0)
$$
which is exactly $P_\zb(z_0)=\beta(z_0)\overline P(z_0)$ for $\beta(z_0)=-\frac2\lambda
D_w\nabla^2v(z_0,z_0,w_0,u_0)(\pz,\pz)(\pz)$.
\end{proof}

In order to study the behaviour of $Q$ near a vanishing point we will use two lemmas.

\begin{lemma}\label{lem:isolate1}
Let $f:U\subset \C\to \C$ be a complex function defined in an open subset $U$. Assume that 
$$
\left|f_\zb\right|\le h|f|
$$
where $h$ is a continuous non-negative function. Assume further that $z=z_0\in U$ is a
zero of $f$. Then either $f\equiv 0$ in a neighborhood $V\subset U$ of $z_0$ or there is
$k\in \N$ such that 
$$
f(z)=(z-z_0)^k \tilde f(z), \ z\in V\,,
$$
where $\tilde f$ is a continuous function with $\tilde f(z_0)\neq 0$.
\end{lemma}

This is \cite[Lemma 2.7.1, pp 75]{J}. Actually we will also use a version of this lemma
on the boundary. So let $D$ denote the unit disk in $\C$ and $D_\pm=\{z\in D \, |\,\, \pm\Im z>0\}$. We then have the following variant of the above lemma.

\begin{lemma}\label{lem:isolate2}
Let $f:\overline D_+\to \C$ be a continuous complex function which is $C^1$ in $D_+$ and
satisfies $f(z)\in \R$ if $z\in\R$. Assume that in $D_+$
$$
\left|f_\zb\right|\le h|f|
$$
where $h$ is a continuous non-negative function on $\overline D$. Assume further that $f(0)=0$. Then either
$f\equiv 0$ in a neighborhood $V$ of $0$ or there is $k\in \N$ such that 
$$
f(z)=z^k \tilde f(z), \ z\in V\,,
$$
where $\tilde f$ is a continuous function with $\tilde f(0)\neq 0$.
\end{lemma}

\begin{proof}
For $c>0$ and $\zeta\in \C$, we denote by $D_c(\zeta)$ the disk in $\C$ of center $\zeta$
and radius $c$. Let $w\in D_R(0)\setminus\{0\}$, $R<1$, and define
$W=D_R(0)\setminus(D_a(0)\cup D_a(w))$, where $a<\min (|w|/2,R-|w|)$. 

Let us extend the definition of $f$ to $D_-$ by $f(z)=\overline{f(\zb)}$. Since $f$ is
real on $D\cap\R$, this gives us a continuous complex function in $\overline D$. If we
extend $h$ by $h(z)=h(\zb)$ we get in $D_-$, $|f_\zb|\le h|f|$. 

If $r\in \N$, we can then define a continuous complex $1$-form by
$$
\phi=\frac{f(z)}{z^r(z-w)}dz
$$
which is $C^1$ in $D_+\cup D_-$. Moreover $d\phi=-\frac{f_\zb}{z^r(z-w)}dz\wedge d\zb$. We
then have the following computation
\begin{align*}
\int_W d\phi&=\int_{W\cap D_+} d\phi+\int_{W\cap D_-} d\phi\\
&=\int_{\partial D_R(0)\cap D_+}\phi-\int_{\partial D_a(0)\cap D_+}\phi- \int_{\partial
D_a(w)\cap D_+}\phi+ \int_{[-1,1]\cap W}\phi\\
&\qquad+\int_{\partial D_R(0)\cap D_-}\phi-\int_{\partial D_a(0)\cap D_-}\phi- \int_{\partial
D_a(w)\cap D_-}\phi- \int_{[-1,1]\cap W}\phi\\
&=\int_{\partial D_R(0)}\phi-\int_{\partial D_a(0)}\phi- \int_{\partial D_a(w)}\phi\,.
\end{align*}

Now, the end of the proof is similar to the one of \cite[Lemma 2.7.1, pp 75]{J} which only
uses the continuity of $f$ and a uniform bound on $h$.
\end{proof}

We can now obtain statement~3 in Theorem~A.

\begin{claim}\label{cl:d}
Either $Q\equiv 0$ on $\overline\Ome$ or $Q$ has isolated zeroes.
Moreover, in the second case, the null directions of $Q$ determine on $\overline\Omega$
two $C^1$ line fields with isolated singularities of negative index.
\end{claim}

\begin{proof}[Proof of Claim~\ref{cl:d}]
Let $x_0$ be a zero of $Q$. Assume that $x_0\in\Ome$ and choose some complex conformal
coordinate near $x_0$ such that $z=0$ is $x_0$. The $(2,0)$-part $\boP=P(z)dz^2$ vanishes
then at $0$. By Claim~\ref{cl:c} and Lemma~\ref{lem:isolate1}, either $P$ vanishes in a
neighborhood of $0$ or $P(z)=z^n\tilde f(z)$ for some $n\in \N$ and $\tilde f(0)\neq 0$.
In the second case, $0$ is an isolated zero of $P$ so is $x_0$ for $Q$. The writing
$z^n\tilde f(z)$ implies that the null directions of $Q$ has negative index around $x_0$.
We have then proved that either $Q\equiv 0$ or it has isolated zeroes in $\Ome$.

If $x_0\in\partial \Ome$, we choose some complex conformal coordinate $z\in\overline D_+$
near $x_0$ such that $z=0$ is $x_0$ and $z\in \R$ correspond to $\partial\Ome$. We then
have $\pz=\frac{\sqrt\lambda}2(\vec\tau+i\vec\eta)$ along $\overline D_+\cap\R$. By
Claim~\ref{cl:b}, this implies
$$
P=Q(\pz,\pz)=\frac\lambda4(Q(\vec\tau,\vec\tau)-Q(\vec\eta,\vec\eta)+2iQ(\vec\tau,\vec\eta))=
\frac\lambda4(Q(\vec\tau,\vec\tau)-Q(\vec\eta,\vec\eta))\in \R.
$$
As above, the conclusion follows using Lemma~\ref{lem:isolate2}.
\end{proof}

Next, we verify statement~2 in Theorem~A.

\begin{claim}\label{cl:e}
If $Q\equiv 0$ on $\Ome$ then $u\in\mathcal C_\alpha$.
\end{claim}

\begin{proof}[Proof of Claim~\ref{cl:e}]
Let $x_0\in \Ome$ and write $\bar v=v_{x_0,\nabla u(x_0),u(x_0)}$. Let $X$ be normal
coordinates at $x_0$. At $x_0$, $u$ and $\bar v$ has a contact of order at least $2$.
Moreover
$$
\Delta (u-\bar v)=f(u)-f(\bar v)=g\times(u-\bar v)
$$
where $g=\int_0^1f'(tu+(1-t)\bar v) dt$. So by Bers Theorem \cite{bers}, either
$u=\bar v$ near $0$ or $(u-\bar v)(X)=q(X)+o(\|X\|^k)$ where $q\not\equiv 0$ is a
homogeneous harmonic polynomial of degree $k\ge 2$. Assume we are in this second case. So
we can write
\begin{align*}
Q_x&=\nabla^2 u(x)-\nabla^2 v_{x,\nabla u(x),u(x)}(x)\\
&=(\nabla^2 u(x))-\nabla^2 \bar v(x))+(\nabla^2 \bar v(x)-\nabla^2 v_{x,\nabla u(x),u(x)}(x))\,.
\end{align*}
In coordinates the hessian operator $\nabla^2$ can be written as an operator which is the
Euclidean hessian $H$ up to a $O(\|X\|)$ error term. So the first term in the above
computation is just $H q(X)+o(\|X\|^{k-2})$. For the second term, we notice that $\bar
v=v_{x,\bar v(x),\nabla \bar v(x)}$ and the map $(w,a)\in T_xM\times\R\mapsto\nabla^2
v_{x,w,a}(x)$ is smooth. So the second term can be evaluated by
$$
\nabla^2 \bar v(x)-\nabla^2 v_{x,u(x),\nabla u(x)}(x)=O(|u-\bar v|(x)+\|\nabla u-\nabla
\bar v\|(x))=O(\|X\|^{k-1})
$$
So $Q_x=H q(X)+o(\|X\|^{k-2})$. As $Q$ vanishes around $x_0$ this implies that $q=0$. So
we get a contradiction and $u=\bar v$ near $x_0$ and by connectedness, $u=\bar v$
everywhere and $\Ome=\Ome_{x_0,\nabla u(x_0),u(x_0)}$.
\end{proof}

\begin{remark}
Let us notice that the argument used in Claim \ref{cl:e} can be applied to prove that $Q$ has isolated zeroes of negative index in $\Ome$ but it seems difficult to do the same at $x_0\in\partial\Ome$ since we do not have a result similar to Bers Theorem at a boundary point.
\end{remark}


\section*{Acknowledgments}
\renewcommand{\thesection}{\arabic{section}}
\renewcommand{\theequation}{\thesection.\arabic{equation}}
\setcounter{equation}{0} \setcounter{maintheorem}{0}

The authors would like to thanks A. Savo, D. Ruiz and P. Mira for their comments and interest in this paper. 

The first author, Jos\'{e} M. Espinar, is partially supported by Spanish MEC-FEDER Grant
MTM2013-43970-P; CNPq-Brazil Grants 405732/2013-9 and 14/2012 - Universal, Grant
302669/2011-6 - Produtividade; FAPERJ Grant 25/2014 - Jovem Cientista de Nosso Estado. 



\begin{thebibliography}{99}

\small

\bibitem{AR1} U. Abresch and H. Rosenberg, A Hopf differential for constant mean curvature
surfaces in $\s ^2 \times \R$ and $\h^2 \times \R$, \emph{Acta Math.}, {\bf 193} (2004)
141--174.

\bibitem{AR2} U. Abresch and H. Rosenberg, Generalized Hopf differentials, \emph{Math.
Contemp.}, {\bf 28} (2005) 1--28.

\bibitem{AftBus} A. Aftalion and J. Busca. Radial symmetry of overdetermined boundary value problems in exterior domains, \emph{Arch. Rat. Mech. Anal.}, {\bf 143} (1998) {\bf no. 2}, 195--206.

\bibitem{AEG1} J.A. Aledo, J.M. Espinar and J.A. G\'{a}lvez, Complete surfaces of constant
curvature in $\h^2 \times \R$ and $\s^2 \times \R$, {\it Calc. Variations \& PDE's}, {\bf
29} (2007), 347--363.

\bibitem{AEG3} J.A. Aledo, J.M. Espinar and J.A. G\'{a}lvez, The Codazzi equation on
surfaces. \emph{Adv. in Math.}, {\bf 224} (2010), 2511--2530.

\bibitem{a} A.D. Alexandrov, Uniqueness theorems for surfaces in the large, I, {\it
Vestnik Leningrad Univ. Math.}, {\bf 11} (1956) 5--17 (in Russian).


\bibitem{bcn} H. Berestycki, L.A. Caffarelli and L. Nirenberg,
Monotonicity for elliptic equations in unbounded Lipschitz domains,
\emph{Comm. Pure. Appl. Math.}, {\bf 50} (1997) 1089--1111.

\bibitem{bers} L. Bers, Remark on an application of pseudo analytic functions, {\it Amer.
J. Math.}, {\bf 78} (1956), 486--496.


\bibitem{Cha} I. Chavel, Eigenvalues in Riemannian geometry. \emph{Academic Press}, New York, 1984.

\bibitem{chern} S.S. Chern, On special $W-$surfaces, \emph{Proc. Amer. Math. Soc.}, {\bf
6} (1955), 783--786. 

\bibitem{Choe} J. Choe, Sufficient conditions for constant mean curvature surfaces to be
round, \emph{Math. Ann.}, {\bf 323} (2002), 143--156.





\bibitem{EFM} J. M. Espinar, A. Farina and L. Mazet, $f-$extremal Domains in Hyperbolic
Space. \emph{Preprint}.

\bibitem{EGR} J. M. Espinar, J. A. G\'alvez and H. Rosenberg, Complete surfaces with
positive extrinsic curvature in product spaces, {\it Comment. Math. Helv.}, {\bf 84}
(2009), 351--386.

\bibitem{EM} J. M. Espinar and J. Mao, Extremal Domains on Hadamard Manifolds. \emph{Preprint}.

\bibitem{fv4} A. Farina, L. Mari and E. Valdinoci, Splitting theorems, symmetry results
and overdetermined problems for Riemannian manifolds. \emph{ Comm. Partial Differential
Equations} 38 (2013), no. 10, 1818–-1862.

\bibitem{fv1} A. Farina and E. Valdinoci, Flattening results for
elliptic PDEs in unbounded domains with applications to
overdetermined problems, \emph{Arch. Rat. Mech. Anal.}, {\bf 195}
(2010) 1025--1058.

\bibitem{fv2} A. Farina and E. Valdinoci, Partially and globally
overdetermined problems of elliptic type, \emph{Adv. Nonlinear
Anal.}, {\bf 1} (2012) 27--45.

\bibitem{fv3} A. Farina and E. Valdinoci, On partially and globally overdetermined
problems of elliptic type, \emph{Am. J. of Math.}, {\bf 135} (6) (2013) 1699--1726.

\bibitem{GM} J.A. G\'{a}lvez and P. Mira, Uniqueness of immersed spheres in
three-manifolds. {\it Preprint}.


\bibitem{HW} P. Hartman, W. Wintner, Umbilical points and $W-$surfaces, {\it Amer. J. Math.}, {\bf 76} (1954), 502--508.

\bibitem{Ho} H. Hopf, Differential Geometry in the large. \emph{Lecture Notes in Math.},
{\bf 1000}, Springer, Berlin, 1983.

\bibitem{J} J. Jost, Two dimensonal Geometric Variational problems, \emph{Pure and Applied
Mathematics} (New York). A Wiley-Interscience Publication. John Wiley \& Sons, Ltd.,
Chichester, 1991.

%
\bibitem{MMPR} W. Meeks, P. Mira, J. P\'{e}rez and A. Ros, Constant mean curvature spheres in homogeneous three-spheres. \emph{Preprint}.

\bibitem{Mira} P. Mira, Overdetermined elliptic problems in topological disks \emph{Preprint}.


\bibitem{mr} R. Molzon, Symmetry and overdetermined boundary value
problems, \emph{Forum Math.}, {\bf 3} (1991) 143--156.

\bibitem{Nit} J. C. C. Nistche, Stationary partitioning of convex bodies. \emph{Arch.
Rational Mech. Anal.}, {\bf 89} (1985), 1-19.

\bibitem{ps} P. Pucci and J. Serrin, The maximum principle. {\it Progress in Nonlinear
Differential Equations and Their Applications}, Birkhauser, Basel, 2007.

\bibitem{RauSav} S. Raulot and A. Savo, On the first eigenvalue of the Dirichlet-to-Neumann operator on forms, \emph{J. Funct. Anal.}, {\bf 262} (2012), 889--914.

\bibitem{Rei} W. Reichel, Radial symmetry for elliptic boundary-value problems on exterior domains. {\it Arch. Rational Mech. Anal.} {\bf 137} (1997) 381--394. 

\bibitem{rs} A. Ros and P. Sicbaldi, Geometry and topology of some
overdetermined elliptic problems, \emph{J. Differential Equations},
{\bf 255} (2013) 951--977.

\bibitem{RRS} A. Ros, D. Ruiz and P. Sicbaldi, A rigidity result for overdetermined
elliptic problems in the plane, Preprint.

\bibitem{RRS2} A. Ros, D. Ruiz and P. Sicbaldi, Solutions to overdetermined elliptic problems in nontrivial exterior domains, Preprint.

\bibitem{Sch} M. Schiffer, Hadamard's formula and variation of domain functions, \emph{Amer. J. Math.}, {\bf 68} (1946), 417--448.

\bibitem{s} J. Serrin, A symmetry problem in potential theory, \emph{Arch. Rational Mech.
Anal.}, {\bf 43} (1971) 304--318.

\bibitem{Shk} V. Shklover, Schiffer problem and isoparametric hypersurfaces, \emph{Revista Mat. Iberoamericana}, {\bf 16} (2000) {\bf no. 3}, 529--569.

\bibitem{Sic} P. Sicbaldi, New extremal domains for the first eigenvalue of the Laplacian
in flat tori, {\it Calc. Var. Partial Differential Equations}, {\bf 37} (2010) 329--344.


\bibitem{Sir} B. Sirakov. Symmetry for exterior elliptic problems and two conjectures in potential theory, \emph{Ann. Inst. H. Poincar\'{e} (C) Nonl. Anal.}, {\bf 18} (2001) {\bf no. 2}, 135--156.


\end{thebibliography}
\end{document}